\documentclass{birkjour}

\newtheorem{thm}{Theorem}[section]
 \newtheorem{cor}[thm]{Corollary}
 \newtheorem{lem}[thm]{Lemma}
 \newtheorem{prop}[thm]{Proposition}
 \theoremstyle{definition}
 
 \theoremstyle{remark}
 
 \newtheorem*{ex}{Example}
 \numberwithin{equation}{section}

\usepackage{ytableau}

\begin{document}

\title[The Right Key of a Tableau]
 {A Direct Way to Find the Right Key of a Semistandard Young Tableau$^1$\footnote{$^1$To be contained in the author's doctoral thesis written under the supervision of Robert A. Proctor.}}

\author[M. J. Willis]{Matthew J. Willis}

\address{Department of Mathematics, University of North Carolina, Chapel Hill, NC 27599, USA}

\email{mjwillis@email.unc.edu}

\begin{abstract}
The right key of a semistandard Young tableau is a tool used to find Demazure characters for $sl_n(\mathbb{C})$.  This paper gives methods to obtain the right and left keys by inspection of the semistandard Young tableau.
\end{abstract}



\maketitle

October 27, 2011

\section{Introduction}

A ``key'' is a semistandard tableau such that the entries in each column also appear in each column to the left.  If a semistandard tableau $T$ is not a key itself, then the ``right (left) key'' of $T$ is a slightly greater (lesser) tableau (via entrywise comparison) that is associated to $T$.

Perhaps the foremost result that uses the notion of the right key of a tableau is a summation formula [6, Thm. 1] for Demazure characters due to Lascoux and Sch{\"u}tzenberger [3].  Recall a well known description of Schur functions:  Given a partition $\lambda$, the Schur function $s_\lambda$ is equal to the sum of the weights of all semistandard tableaux of shape $\lambda$.  More generally, given a highest weight $\mu$ and a permutation $w$, the Demazure character (or ``key polynomial'') determined by $\mu$ and $w$ of a ``Borel subalgebra'' of $sl_n(\mathbb{C})$ is equal to the sum of the weights of all semistandard tableaux of shape $\mu$ whose right keys are less than or equal to the key corresponding to $w$.

We now describe a method to quickly write down the right key of a given semistandard tableau $T$.  To do so we need only one special definition beyond the usual terminology.  Given a sequence $(x_1,x_2,x_3,...)$, its \textit{earliest weakly increasing subsequence} (EWIS) is $(x_{i_1},x_{i_2},x_{i_3},...)$, where $i_1 = 1$ and for $j>1$ the index $i_j$ is the smallest index such that $x_{i_j} \geq x_{i_{j-1}}$.

Here is an English description of our method that a non-mathematician could apply:  Draw an empty Young diagram that has the same shape as $T$; this diagram will receive the entries of the right key of $T$.  Viewing the bottom entries of the columns of $T$ from left to right as forming a sequence, use two index fingers to find the EWIS of this sequence.  Whenever an entry is added to the EWIS, put a dot above it.  When this EWIS ends, write its last member in the diagram for the right key of $T$ in the lowest available box in the leftmost available column.  Next, repeat the process as if the boxes with dots are no longer part of $T$.  Once every box in $T$ has a dot, the leftmost column of the right key has been formed.  To find the entries of the next column in the right key: ignore the leftmost column of $T$, erase the dots in the remaining boxes, and repeat the above process.  Continue in this manner until the Young diagram has been filled with the entries of the right key of $T$.

It can be seen that if there is more than one column of a given length, upon reaching this length one may ignore all but the rightmost of these columns.  Also, erasing the dots is unnecessary if there is one color of ink available for each distinct column length.

\begin{ex}
Let $T$ be the semistandard tableau shown in Figure 1.  First we find the EWISs that begin in the first column of $T$.  To display all six EWISs with just one figure, we use the over-decorations $\dot{\textcolor[rgb]{1.00,1.00,1.00}{1}}$, $\bar{\textcolor[rgb]{1.00,1.00,1.00}{1}}$, $\tilde{\textcolor[rgb]{1.00,1.00,1.00}{1}}$, $\acute{\textcolor[rgb]{1.00,1.00,1.00}{1}}$, $\hat{\textcolor[rgb]{1.00,1.00,1.00}{1}}$ and $\check{\textcolor[rgb]{1.00,1.00,1.00}{1}}$ to indicate their successive creations and removals.
\end{ex}

\begin{figure}[h]
\caption{}

\ytableausetup{centertableaux, boxsize = 2em}
\begin{ytableau}
\check{1} & \check{1} & \hat{3} & \hat{4} & \acute{6} \\
\hat{2} & \hat{3} & \acute{5} & \tilde{7} & \dot{9} \\
\acute{4} & \tilde{5} & \tilde{6} & \bar{8} \\
\tilde{5} & \bar{7} & \dot{9} \\
\bar{7} \\
\dot{8}
\end{ytableau}
\hspace{25mm}
\begin{ytableau}
1 & 6 & 6 & 6 & 6 \\
4 & 7 & 7 & 7 & 9 \\
6 & 8 & 8 & 9 \\
7 & 9 & 9 \\
8 \\
9
\end{ytableau}

\end{figure}

So the six EWISs are (8,9,9), (7,7,8), (5,5,6,7), (4,5,6), (2,3,3,4), and (1,1).  Hence the entries of the 1st column of the right key of $T$ are 9, 8, 7, 6, 4, and 1.  To find the EWISs for the later columns, we repeat the process (after erasing the over-decorations) as if the first column is no longer part of $T$.  The four EWISs for the rightmost length 4 column are (9,9), (6,8), (5,7), and (3,4,6).  Continuing the procedure for the length 3 column, the three EWISs are (8,9), (7), and (4,6).  For the length 2 column (and in fact the rightmost column of any SSYT), the EWISs are just the entries in the column.  Writing in the last members from each of these EWISs, we obtain the right key of T, which is also displayed in Figure 1.

This paper's main result, Theorem 4.5, states that the output of this method is in fact the right key of $T$.  The proof uses the jeu de taquin description of the right key presented in [2].

\section{Scanning Tableau}

Now we describe the above procedure with mathematical notation.  A \textit{shape} $\zeta$ is a sequence $\zeta = (\zeta_1 , \zeta_2, ... , \zeta_k)$ of positive integers such that $\zeta_1 \geq \zeta_2 \geq ... \geq \zeta_k$ for some $k \geq 0$.  Each member of a shape is a \textit{(column) length}. The \textit{Young diagram of $\zeta$} is the diagram with $\zeta_i$ top justified boxes in column $i$ for $1 \leq i \leq k$.  A \textit{semistandard tableau of shape $\zeta$} is a filling of the Young diagram of $\zeta$ with positive integers such that the entries are strictly increasing down the columns and weakly increasing across the rows.  Henceforth we more simply write \textit{tableau (of shape $\zeta$}).  (For Lie theoretical purposes, one would bound both the lengths of a shape and the tableau entries by some fixed $n \geq 1$.)

Fix a tableau $T$ of shape $\zeta$ with $k$ columns.  Let $T_i$ denote the $i$th column of $T$ for $1 \leq i \leq k$.  The tableau $T$ is a \textit{key} if the entries in $T_i$ also appear in $T_{i-1}$ for $1 < i \leq k$.  Given a column $C$, let $l(C)$ denote its length.  So $l(T_i) = \zeta_i$.  Define $b_i$ to be the bottom entry of $T_i$ for $1 \leq i \leq k$.  Let $e = (e_1,...,e_m)$ for some $m \leq k$ denote the EWIS obtained from $b = (b_1,...,b_k)$.  Let $d = \{ d_1, ... , d_m \}$ denote the collection of boxes in the Young diagram of $T$ that respectively contain the entries of $e$.  Define $\setminus$ to be the operator that removes a specified set of boxes (and their entries) from a diagram or tableau.  It is easy to see that $T \setminus d$ is a tableau.

Write $x \bigodot C$ to notate the result of attaching a single entry $x$ to the bottom of a column $C$, and $C \bigoplus T$ to notate the result of prepending a column $C$ to the left side of $T$.  The following doubly recursive definition specifies a tableau $S(T)$ that has the same shape as $T$;  here $S_1(T)$ denotes the first column of $S(T)$ and $S_{1B}(T)$ denotes the bottom entry of $S_1(T)$.  Set $S_{1B}(T):= e_m$ (if $T$ is nonempty; otherwise, for terminating purposes set $S_{1B}(T):= \emptyset$).  Set $S_1(T):= S_{1B}(T) \bigodot S_1(T \setminus d)$.  Define $S(T):= S_1(T) \bigoplus S(T \setminus T_1)$.  We call $S(T)$ the \textit{scanning tableau} of $T$.  The fact that $S(T)$ is a key will follow from the main result of this paper.

\section{The Right Key}

In this section we present a common definition of the right key of a tableau, as can be found in Appendix 5 of [2].  We also present a specific method for calculating the right key that uses this definition.  Given a shape $\zeta$ with $k$ lengths and a shape $\eta$ with $l \leq k$ lengths such that $\eta_i \leq \zeta_i$ for $1 \leq i \leq l$, the \textit{skew diagram of shape $\zeta \setminus \eta$} is obtained by removing the diagram of $\eta$ from the diagram of $\zeta$.  The sequence of column lengths of $\zeta \setminus \eta$ is $(\zeta_1 - \eta_1, ... , \zeta_l - \eta_l, \zeta_{l+1}, ... , \zeta_k)$.  Some of the lengths in a skew diagram may be zero.  A \textit{skew tableau} is a filling of a skew diagram with positive integers such that the entries are strictly increasing down the columns and weakly increasing along the rows.

Now we assume familiarity with the jeu de taquin (JDT) process,  as in (for example) Chapter 1 of [2].  The \textit{rectification} of a skew tableau [2, p. 15] is the tableau that results from successively applying ``JDT slides'' to the ``inside corners'' of the skew tableau (in any order) until there are no inside corners left.  It is known that the rectification of a skew tableau is unaffected by any JDT slide or ``reverse JDT (RJDT) slide''.

\begin{lem}
 Let $U$ be the skew tableau of some shape $\zeta \setminus \eta$ that is formed by arranging the $k$ columns $U_1, ... , U_k$ according to $\eta$.  Let $1 \leq l \leq k$ and $d > 0$.  Using RJDT slides, one may shift all of the entries of each of $U_1 , ... , U_l$ down $d$ rows without otherwise modifying $U$.
\end{lem}

\begin{proof}
As $j$ runs from 1 to $l$, successively perform $d$ RJDT slides beginning directly under $U_j$.  It can be seen that for $2 \leq j \leq l$, each of these RJDT slides will pull each entry of $U_j$ down one row (without affecting the columns that have already been pulled down).
\end{proof}

A skew tableau is \textit{frank} if its lengths are a rearrangement of the lengths of the shape of its rectification.

\begin{prop}[]
\textbf{[2, p. 208]}  Given a tableau $T$ and a skew diagram $\zeta \setminus \eta$ whose lengths are a rearrangement of the lengths of $T$, there is a unique skew tableau on $\zeta \setminus \eta$ that rectifies to $T$.
\end{prop}

\begin{cor}
The rightmost column of a frank skew tableau that rectifies to $T$ depends only upon the length of that column.
\end{cor}

Fix a shape $\zeta$ with $k$ lengths.  Let $T$ be a tableau of shape $\zeta$.  For $1 \leq i \leq k$, define $R_i(T)$ to be the rightmost column of any frank skew tableau whose rightmost column has length $\zeta_i$ and that rectifies to $T$.  The \textit{right key of} $T$ is $R(T) := \bigoplus_{i=1}^k R_i(T)$.  Note that $R(T)$ has shape $\zeta$.  Also, it is known [2, App. A.5 Cor. 2] that $R(T)$ is a key.  Before specifying a method to obtain the individual columns of $R(T)$, we define an operation that will be iterated to produce each of them.

Let $U$ be a skew tableau of some shape $\zeta \setminus \eta$ formed from the $k$ columns $U_1 , ... , U_k$, where $\eta$ has $l \leq k$ lengths.  For any $l < j < k$, set $x_j := l(U_j) - l(U_{j+1})$ and define $d_j$ to be the number of boxes in $U_{j-1}$ that are attached to boxes in $U_j$.  For such $j$, the process of successively shifting $U_1, ... , U_{j-1}$ down $d_j$ rows and then successively performing $x_j$ RJDT slides on the outside corner of $U_j \bigoplus U_{j+1}$ is called the \textit{jth length swap (on $U$)}.  Note that the lengths of the resulting skew tableau are $(l(U_1), ... , l(U_{j-1}), l(U_{j+1}), l(U_j), l(U_{j+2}), ... , l(U_k))$.

\begin{lem}
Let $U$ be a frank skew tableau of shape $\zeta \setminus \eta$ that rectifies to a tableau $T$.  For $l < j < k$, the $j$th length swap produces a frank skew tableau $U'$ that also rectifies to $T$.
\end{lem}

\begin{proof}
Rectification is preserved by RJDT slides.  Since $U$ is frank and the lengths of $U'$ are a rearrangement of the lengths of $U$, we have that $U'$ is also frank.
\end{proof}

\begin{prop}
Let $T$ be a tableau of shape $\zeta$ with $k$ columns and let $1 \leq i < k$.  Successively performing the $i$th, $(i+1)$st, ... , $(k-1)$st length swaps produces a frank skew tableau with lengths $(\zeta_1, ... , \zeta_{i-1} , \zeta_{i+1} , ... , \zeta_k, \zeta_i)$  that rectifies to $T$.  Its rightmost column is $R_i(T)$.
\end{prop}

\begin{proof}
Apply Lemma 3.4 as $j$ runs from $i$ to $k-1$.
\end{proof}

\section{Main Result}

The method of producing $R(T)$ in this section parallels the construction of $S(T)$ in Section 2.  Let $\zeta$ be a shape with $k$ lengths.  Fix a tableau $T$ of shape $\zeta$.

\begin{lem}
$R(T) = R_1(T) \bigoplus R(T \setminus T_1)$.
\end{lem}

\begin{proof}
Proposition 3.5 implies that the calculation of $R_i(T)$ does not involve the first $i-1$ columns of $T$.  Therefore the calculations for $\bigoplus_{i=2}^k R_i(T)$ are independent of $T_1$.
\end{proof}

We now focus on producing $R_1(T)$.  Define $T^{(1)} := T$ and for $1 \leq j \leq k-1$ define $T^{(j+1)}$ to be the skew tableau obtained by performing the $j$th length swap on $T^{(j)}$.  The shape of $T^{(j)}$ has lengths $(\zeta_2, ... , \zeta_j, \zeta_1, \zeta_{j+1}, ... , \zeta_k)$. Therefore $T_k^{(k)} = R_1(T)$.  Let $R_{1B}(T)$ denote the bottom entry of $R_1(T)$.  Recall that $b_i$ is the bottom entry of $T_i$ for $1 \leq i \leq k$, and that $b = (b_1, ... , b_k)$.  For $1 \leq j, h \leq k$, let $b_h^{(j)}$ denote the bottom entry of $T_h^{(j)}$.  Since $R_1(T) = T_k^{(k)}$, we have that $R_{1B}(T) = b_k^{(k)}$.  Note that if $h > j$, then $T_h^{(j)} = T_h$ and hence $b_h^{(j)} = b_h$.  For $1 \leq j \leq k-1$, the entry $b_{j+1}^{(j+1)}$ produced by the $j$th length swap on $T^{(j)}$ is one of two possible entries:

\begin{lem}
Let $1 \leq j \leq k-1$. \\
(i)  If $b_{j+1}^{(j)} \geq b_j^{(j)}$, then $b_{j+1}^{(j+1)} = b_{j+1}^{(j)}$. \\
(ii)  Otherwise, $b_{j+1}^{(j+1)} = b_j^{(j)}$.
\end{lem}

\begin{proof}
Let $1 \leq j \leq k-1$.  The difference in lengths of the two rightmost columns involved in the $j$th length swap on $T^{(j)}$ is $x_j =  \zeta_1 - \zeta_{j+1}$.  First suppose $b_{j+1}^{(j)} \geq b_j^{(j)}$.  In this case, by the semistandard conditions $b_{j+1}^{(j)}$ is greater than every entry in $T_{j}^{(j)} \bigoplus T_{j+1}^{(j)}$ except perhaps $b_j^{(j)}$.  Thus the first $x_j-1$ RJDT slides of the $j$th length swap will each pull $b_{j+1}^{(j)}$ down one row.  Then the $x_j$th RJDT slide will compare $b_j^{(j)}$ to $b_{j+1}^{(j)}$ and also pull $b_{j+1}^{(j)}$ down. Hence $b_{j+1}^{(j+1)} = b_{j+1}^{(j)}$.  Otherwise, $b_{j+1}^{(j)} < b_j^{(j)}$.  In this case, $b_j^{(j)}$ is greater than every entry in $T_{j}^{(j)} \bigoplus T_{j+1}^{(j)}$.  As a result, the $x_j$th RJDT slide will move $b_j^{(j)}$ to the right regardless of what it is compared to.  Thus $b_{j+1}^{(j+1)} = b_j^{(j)}$.
\end{proof}

Note that {\sl (ii)} happens precisely when the final RJDT slide of the $j$th length swap moves the bottom entry of $T_j^{(j)}$ across to become the bottom entry of $T_{j+1}^{(j+1)}$.  In this case, the EWIS $e$ of $b$ ``skips'' $T_{j+1}$.  The complementary possibility {\sl (i)} occurs precisely when a new member is appended to $e$.  Hence:

\begin{cor}
$R_{1B}(T) = S_{1B}(T)$.
\end{cor}

\begin{proof}
Apply these two observations as $j$ run from 1 to $k-1$ in Lemma 4.2 to see that $b_k^{(k)}$ is the last member of $e$.
\end{proof}

Recall that $d$ denotes the collection of boxes containing $e$:

\begin{lem}
$R_1(T) = R_{1B}(T) \bigodot R_1(T \setminus d)$.
\end{lem}

\begin{proof}
Let $1 \leq j \leq k-1$.  We need to show that the entries of $T_{j+1}^{(j+1)}$ above $b_{j+1}^{(j+1)}$ determined by the $j$th length swap on $T^{(j)}$ are unaffected by the presence of the entries of $e$ in $T_j^{(j)}$.  (These entries are initially located at $d$ in $T^{(1)} = T$.) It is clear that:  $(*)$ For $1 \leq j \leq k-1$, the presence of $b_j^{(j)}$ will not affect whether or not the entries in $T_j^{(j)}$ above it will be pulled to the right during the $j$th length swap.  Since $e_1 = b_1^{(1)}$, it can be safely removed from $T$ (and its successor skew tableaux).

From Lemma 4.2, all but the first member of $e$ arise as $b_{j+1}^{(j)}$ for some $1 \leq j \leq k-1$.  Suppose $1 \leq j \leq k - 1$ and $b_{j+1}^{(j)} = b_{j+1}^{(j+1)}$.  By the proof of Lemma 4.2 we have that $b_{j+1}^{(j)}$ was pulled straight down from the bottom of $T_{j+1}^{(j)}$ to the bottom of $T_{j+1}^{(j+1)}$.  So the first move in every RJDT slide executed in the $j$th length swap pulls $b_{j+1}^{(j)}$ down 1 row.  Then the remainder of the RJDT slide continues as it would have if $b_{j+1}^{(j)}$ was not originally part of $T^{(j)}$.  Now $b_{j+1}^{(j)} = b_{j+1}^{(j+1)}$, so by $(*)$ it can be safely removed from $T$ (and its successor skew tableaux).  Similar reasoning shows that $b_k^{(k)}$ can be safely removed when $b_k^{(k-1)} = b_k^{(k)}$.  So we see that the presence of the entries in $e$ does not affect the determination of $R_1(T) \setminus R_{1B}(T)$.  Thus once $R_{1B}(T)$ is found, we can remove all of $d$ (and $e$) from $T$ and then find the rest of $R_1(T)$ iteratively without the entries of $e$ being present.
\end{proof}

\begin{thm}
Let $T$ be a tableau.  The right key $R(T)$ is equal to the scanning tableau $S(T)$.
\end{thm}

\begin{proof}
The combination of Corollary 4.3, Lemma 4.4, and Lemma 4.1 agrees with the doubly recursive definition of $S(T)$.
\end{proof}

\section{The Left Key}

Here is a method to write down the left key of a given semistandard tableau $T$:  Draw an empty Young diagram that has the same shape as $T$.  Let $k$ be the number of columns of $T$.  Let $a = (a_1, ... , a_k)$ be the sequence defined as follows:  Let $a_1$ be the bottom entry of the $k$th column. For $1 < i \leq k$, let $a_i$ be the largest entry in the $(k+1-i)$th column that is less than or equal to $a_{i-1}$.  Note that by the semistandard row condition, at least one such entry is guaranteed to exist.  Put a dot above each entry of $a$.  Place $a_k$ in the lowest available box in the rightmost available column of the Young diagram.  Next, repeat the process as if the boxes with dots and all boxes below them are no longer part of $T$.  Once such a sequence has been found for each entry in the rightmost column of $T$, the rightmost column of the left key of $T$ has been determined.  To find the next column in the left key, ignore the rightmost column of $T$, erase all the dots, and repeat the above process.  Continue in this manner until the entire Young diagram has been filled.

The proof that the tableau produced by this method is the left key of $T$ is similar to the proof of the analogous statement concerning the right key given in this paper.  In fact the proof is simpler, since the sequence produced above necessarily contains exactly one entry from each column to the left of its starting point.

\section{Previous Methods}

Lascoux and Sch\"{u}tzenberger introduced [3] the notions of right and left key in Definition 2.9 and developed their relation to Demazure characters on pp. 132 - 138.  The paper [6] by Reiner and Shimozono explicitly stated the summation formula for the ``key polynomial'' in Theorem 1.  Fulton's book [2] appeared in 1997.  In [4] Lenart stated (on p. 280) that one can use ``reverse column insertion'' to produce the left key of a tableau, restates [6, Thm. 1] as Theorem 4.1, and discusses how this result may be proved using material from [3].  In [1], Aval showed how to compute the left key of a tableau using manipulations of ``sign matrices''.  Then he obtained the right key from the left key by using the notion of ``complement'' of a Young tableau.  Mason presented a method in [5] to produce the right key of a tableau using ``semiskyline augmented fillings''.  Both [1] and [5] had larger goals than key computations that led to these procedures.

The author would like to thank Sarah Mason for her remarks and encouragement.  Reading [5] led to the creation of the method presented in this paper, whose original proof referred to her Corollary 5.1.  The author would also like to thank Vic Reiner for suggesting the possibility of the more direct proof given in this paper.  In addition, both Mark Shimozono and Jim Haglund provided valuable perspective to the author's advisor through conversations.

Finally, the author would like to thank his advisor Bob Proctor for suggesting many of the aspects of this presentation and for his repeated encouragement.  Among others, these expositional aspects include the English description of the scanning method, notating shapes with column lengths, and structuring the proof of the main result to parallel the scanning method.

\end{document}